\documentclass[a4paper,11pt]{article}
\usepackage[utf8x]{inputenc}
\usepackage{amsmath,amsfonts,amssymb,amstext,amsthm,amscd,mathrsfs}
\usepackage{epsfig}
\usepackage[all]{xy}

\def\C{{\bf \mathbb{C}}}

\def\P{{\bf \mathbb{P}}}

\def \Na{{\mathcal{N}}}

\newtheorem{definition}{Definition}[section]

\newtheorem{ex}[definition]{Example}

\newenvironment{example}{\begin{ex}\em}{\end{ex}}
\newtheorem{Theorem}[definition]{Theorem}
\newtheorem{corollary}[definition]{Corollary}

\newtheorem{proposition}[definition]{Proposition}
\newtheorem{lema}[definition]{Lemma}

\title{On the Nash Modification of a Germ of Complex Analytic Singularity}
\author{Arturo Giles Flores \footnote{Research partially supported by CONACYT (Mexico) grant 221635 }}

\begin{document}

\maketitle

\begin{abstract}
 {\scriptsize For a germ $(X,0) \subset (\C^n,0)$ of reduced, equidimensional complex analytic singularity its Nash modification can be constructed
  as an analytic subvariety $ Z \subset \C^n \times G(k,n)$. We give a characterization of the subvarieties  of $\C^n \times G(k,n)$ that
  are the Nash modification of its image under the projection to $\C^n$. This result generalizes the characterization of conormal varieties
  as Legendrian subvarieties of $\C^n \times \check{\P}^{n-1}$ with its canonical contact structure. As a by-product we define the $d$-conormal
  space of $(X,0)$ for any $d \in \{k, \ldots, n-1\}$ which is a generalization of both the Nash modification and the conormal variety of $(X,0)$. }
 \end{abstract}

\section{Introduction}

   	For a germ of analytic singularity $(X,0)\subset (\C^n,0)$ the set of limits of tangent spaces plays a big role
in the study of equisingularity. If $(X,0)$ is a reduced and irreducible germ of analytic singularity of pure dimension $d$, this set is obtained
as the preimage $\nu^{-1}(0)$ of the Nash modification $\nu:\Na X \to X$. It is then a subvariety of the Grasmannian $G(d,n)$ of $d$-planes 
of $\C^n$ and so has the structure of a projective algebraic variety.\\

      When $X$ is a hypersurface the Grassmannian $G(d,n)$ is a projective space $\check{\P}^{n-1}$ and the 
  set $\nu^{-1}(0)$ can be described via projective duality by a finite family of subcones of the tangent cone $C_{X,0}$,
  which includes all of the irreducible components, known as the aur\'eole of the singularity. \cite[Thm 2.1.1 \& Coro 2.1.3]{L-T2} \\
  
      The generalization of this result to germs of arbitrary codimension needs to replace the Nash modification $\Na X$ by
   the conormal space $C(X)$.  Recall that the conormal space of $X$ in $\C^n$ is an analytic space $C(X) \subset X \times \check{\P}^{n-1}$,
  together with a proper analytic map $\kappa: C(X)\to X$, where the fiber over a smooth point  
  $x \in X$ is the set of tangent hyperplanes to $X$ at $x$, that is the hyperplanes $H \in \check{\P}^{n-1}$ 
  containing the direction of the tangent space $T_xX$. We are then able to once again describe the set of limits of tangent
  hyperplanes  via the aur\'eole and projective duality. See proposition  \cite[pg. 378-381]{Te1} \\ 
  
    What is so useful about this change from tangent spaces to tangent hyperplanes is that eventhough the space $C(X)$ depends on the embedding
  there is a ``numerical" characterization ( in terms of the dimension of a fiber) of the Whitney conditions via the normal/conormal diagram \cite[Chapter 5, Thm 1.2]{Te1} 
  and in theory it is possible to recover the fiber of  the Nash modification (which does not depend on the embedding) from the conormal fiber.\\ 
  
  	The idea is that every limit of tangent hyperplanes $H \in \kappa^{-1}(0)$ contains a limit of tangent spaces $T \in \nu^{-1}(0)$, and so to each such 
$T$ there corresponds, via projective duality, a linear subspace $\check{\P}^{n-d-1} \subset \kappa^{-1}(0) \subset \check{\P}^{n-1}$. This means we 
have to look for linear subspaces of the right dimension contained in the conormal fiber and take their projective duals.\\

	The problem is that not every $T$ obtained this way is a limit of tangent spaces, and it is a simple dimensionality question. Take for instance 
a germ of surface $(S,0)\subset (\C^5,0)$ with an exceptional tangent. According to what we just said each limit of tangent planes $T$ corresponds
to a $\check{\P}^2\subset \kappa^{-1}(0)\subset \check{\P}^4$.

	But the existence of the exceptional tangent tells us that the projective dual of this point of $\P^4$ is contained
in $\kappa^{-1}(0)$. Its projective dual is a $\check{\P}^3$, and so inside it we have a $G(2,3)$ (dimension 2) of possible limits 
of tangent spaces. But they can't all be limits of tangent spaces because we know that the dimesion of $\nu^{-1}(0)$ is at most 1!!!!!!!
And even in a simple case like this we do not know how to distinguish the ones that are limits of tangent spaces from the ones that are not.
\emph{More generally we do not know the size of the contribution of an exceptional cone to the Nash fiber.}\\	

	One of the key results that made working with the conormal easier than with the Nash modification is that conormal varieties can 
be characterized as Legendrian subvarieties of projectivized cotangent spaces with their canonical contact structure.  In this spirit 
we try to characterize analytic subvarieties $Z$ of $\C^n \times G(d,n)$ such that:
\begin{enumerate}
\item $Z$ has dimension $d$.
\item Its image (by the projection) $X$ in $\C^n$ has dimension $d$.
\item $Z$ is the Nash modification of $X$ 
\end{enumerate}
    
     In order to do this we define an analytic $k$-plane distribution on $\C^n \times G(d,n)$ locally defined by a system of analytic forms 
   and look at the corresponding integral subvarieties. Even though we want to find subvarieties $Z$ of dimension $d$, there are subvarieties of 
   dimension greater than $d$ that are compatible with the distribution in the sense that  for every smooth point $p \in Z$ we have that
   the tangent space $T_pZ$ is contained in the corresponding $k$-plane $\mathcal{H}_p$ determined by the distribution. \\
   
      However, if $X\subset \C^n$ is of dimension $k\leq d$ then we can define an analytic subvariety of $\C^n \times G(d,n)$, that generalizes both 
    the Nash modification $\Na X$ and the conormal space $C(X)$ via the limits of tangent $d$-planes. Zak works with this kind of spaces in 
    his book \cite{Zak} but only in the case of projective varieties and calls them higher order Gauss maps. \\
    
    	\section{The $k$-plane distribution on  $\C^n \times G(d,n)$}	

     		Let us first recall that one of the ways of defining analytic charts for the Grassmannian $G(d,n)$ is to view its points as 
     graphs of linear maps defined on a fixed $d$-dimensional subspace of $\C^n$ and taking values in another fixed $(n-d)$-subspace
     of $\C^n$, where these two fixed subspaces are transversal. This is done as follows. 
     		
		Fix a point $W_0 \in G(d,n)$ and a $n-d$ linear subspace $W_1 \subset \C^n$ such that 
		\[\C^n= W_0 \oplus W_1\] 
	For every linear map $L \in \mathrm{Hom}_{\C}(W_0,W_1)$ we have that its graph in $W_0 \times W_1= \C^n$ 
	is a linear subspace $W$ of dimension $d$, that is, a point in $G(d,n)$.  Moreover, we have that $W \in G(d,n)$ is the graph
	of one such linear map $L$ if and only if $W$ is transversal to $W_1$. \\
	
	  Consider the open subset of the Grasmannian
	  \[G_d^0(n,W_1):=\{W \in G(d,n) \, | \, W \pitchfork W_1\}\]
	  and note that it contains $W_0$. If we denote by $\pi_j$ the linear projection from $\C^n$ to $W_j$ then 
	  we have a bijection
	  \begin{align*}
	     \Phi_{W_0,W_1}:G_d^0(n,W_1) & \longrightarrow \mathrm{Hom}_{\C}(W_0,W_1) \\
	            W & \longmapsto L:=\pi_1\circ (\pi_0|_W)^{-1}:W_0 \to W_1 
	  \end{align*}
          Indeed, for every $W \in G_d^0(n,W_1)$ the restriction map $\pi_0|_W:W \to W_0$ is a linear isomorphism
          and the $L$ thus defined has $W$ as its graph. The collection of the charts $\Phi_{W_0,W_1}$ , when $(W_0,W_1)$
          runs over the set of all direct sum decompositions of $\C^n$, with $W_0$ of dimension d, is an analytic atlas for $G(d,n)$.
          Note that to cover $G(d,n)$ it is enough to consider the charts corresponding to all the coordinate $d-$planes with their
          corresponding complementary coordinate $(n-d)-$planes. (See \cite{Pic}) \\
          
           	To better understand the construction of the $k$-plane distribution on $\C^n \times G(d,n)$ let us first recall the canonical contact 
	structure on the projectivized cotangent bundle $\P T^*\C^n = \C^n \times \check{\P}^{n-1}$ with coordinate system $(x_1,\ldots,x_n), [a_1:\cdots:a_n]$.
	If we look at the chart $\varphi_1: U_1 \to \C^{2n-1}$ where $a_1 \neq 0$ 
	 \[(x_1,\ldots,x_n),[a_1:\cdots:a_n] \mapsto \left(x_1,\ldots,x_n,\frac{a_2}{a_1},\ldots, \frac{a_n}{a_1}\right) \]
	 then the hyperplane of the tangent space $T_{\vec{x},[a]}\P T^*\C^n $ chosen by this distribution is given by the kernel of the 1-form
	 \[ dx_1+ \frac{a_2}{a_1}dx_2+ \cdots + \frac{a_n}{a_1}dx_n \; (*)\]
	 But if we identify the tangent space $T_{\vec{x},[a]}\P T^*\C^n $ with the product of tangent spaces $T_{\vec{x}}\C^n \times T_{[a]}\check{\P}^{n-1}$
	 then the kernel $H_{\vec{x},[a]}$ of $(*)$ is identified with $\widetilde{H} \times T_{[a]}\check{\P}^{n-1}$ where $\widetilde{H} \subset \C^n$
	 is the hyperplane determined by the point $[a] \in \check{\P}^{n-1}$.
	 
	 \begin{definition}
	    On the $n+d(n-d)$ dimensional analytic manifold $\C^n \times G(d,n)$ we define a $d+d(n-d)$-plane distribution as follows.
	 Let $(\vec{z},W)$  be a point $ \C^n \times G(d,n)$ and identify its tangent space with the product of tangent spaces $T_{\vec{z}}\C^n \times T_WG(d,n)= \C^n \times T_WG(d,n)$. Then the plane given by the distribution at this point is:
	 \[\mathcal{H}(\vec{z},W):= W \times T_WG(d,n)\]
	 \end{definition}

   \begin{proposition} The distribution $\mathcal{H}$ is locally defined by the kernel of a system of analytic 1-forms of $ \C^n \times G(d,n)$.
   \end{proposition}
   \begin{proof}
      Recall that it is enough to consider charts of the form $\C^n \times G_d^0(n,W_1)$ where $W_1$ is a coordinate (n-d)-plane of $\C^n$,
    and $W_0$ the corresponding ``complementary'' coordinate $d-plane$. To simplify notation and without loss of generality we will assume 
    $W_0=\left< \vec{e_1}, \ldots, \vec{e_d}\right>$ and $W_1= \left< \vec{e_{d+1}}, \ldots, \vec{e_n}\right>$.
    
    	Now from the Grassmannian chart 
	 \begin{align*}
	     \Phi_{W_0,W_1}:G_d^0(n,W_1) & \longrightarrow \mathrm{Hom}_{\C}(W_0,W_1) \\
	            W & \longmapsto L:=\pi_1\circ (\pi_0|_W)^{-1}:W_0 \to W_1 
	  \end{align*}
	  and after identifying each linear map $L \in \mathrm{Hom}_{\C}(W_0,W_1)$ with the corres-ponding $(n-d)\times d$ matrix 
	  with respect to the basis previously established we obtain the chart of $\C^n \times G(d,n)$ given by:
	  \begin{align*}
	     \Psi_{W_0,W_1}:\C^n \times G_d^0(n,W_1) & \longrightarrow \C^n \times \C^{d(n-d)} \\
	            (z_1,\ldots,z_n), W & \longmapsto (z_1,\ldots,z_n,a_{ij}), \, i=1,\ldots,n-d; \, j=1,\ldots, d
	  \end{align*}
 	  where $W=\left<\vec{e_1} + L(\vec{e_1}), \ldots, \vec{e_d} + L(\vec{e_d})\right> $ is the graph of the corresponding linear map 
	  $L= \Phi_{W_0,W_1}(W) \in \mathrm{Hom}_{\C}(W_0,W_1)$.\\
	  
	  In this chart we can define the following system of analytic 1-forms:
	  \[  \left( \begin{array}{c}  dz_{d+1} \\ dz_{d+2} \\ \vdots \\ dz_n \end{array} \right) = 
     \left( \begin{array}{ccc}  a_{11} & \cdots & a_{1d} \\
                                        a_{21} & \cdots & a_{2d} \\ \vdots & \vdots & \vdots \\
                                        a_{(n-d)1} & \cdots & a_{(n-d)d} \end{array} \right) 
          \left( \begin{array}{c}  dz_1 \\ dz_2 \\ \vdots \\ dz_d \end{array} \right) \]
          whose kernel at a point $(z_1,\ldots,z_n),W \in \C^n\times G(d,n)$ is 
           \[\mathcal{H}(\vec{z},W)=W \times T_WG(d,n) \subset \C^n \times T_WG(d,n)=T_{\overline{z},W}(\C^n\times G(d,n))\]
   \end{proof}
	 
\section{Integral Subvarieties}

	Once we defined the $k-plane$ distribution the next step is to characterize, or find  the corresponding integral subvarieties.
	
\begin{definition}
      The analytic subvariety $Z \subset \C^n \times G(d,n)$ is an integral subvariety of $\left(\C^n \times G(d,n), \mathcal{H}\right)$ if
    for every smooth point $(\vec{z},W) \in Z$ we have that $T_{\vec{z},W}Z \subset \mathcal{H}(\vec{z},W) $.
\end{definition}

     The definition of the distribution puts a restriction on both the dimension of the integral subvariety $Z$ and the dimension of its projection
     on $\C^n$. 
     
 \begin{proposition}\label{Cotadimensional}
        Let $\pi: \C^n \times G(d,n) \to \C^n$ be the projection onto $\C^n$. If $Z \subset \C^n \times G(d,n)$ is an integral subvariety of 
        $\left(\C^n \times G(d,n), \mathcal{H}\right)$ then $t:=\mathrm{dim}\, \pi(Z) \leq d$ and $\mathrm{dim}\, Z \leq t + (d-t)(n-d)$.
 \end{proposition}
 \begin{proof}
    Just by looking at the definition of integral subvariety we have that $T_{p,W}Z \subset \mathcal{H}(p,W) $ and this implies
  that $\mathrm{dim}\, Z \leq d+d(n-d)$.  Since $\pi$ is a proper map $\pi(Z)\subset \C^n$ is an analytic subvariety, and the restriction 
  $\pi:Z \to \pi(Z)$ is generically submersive. Then, for any (sufficiently general) point $(p,W) \in Z^0$ with smooth image $p \in \pi(Z^0)$  we have
  that
  \[T_p\pi(Z) \subset D_p\pi(\mathcal{H}(p,W))=W\]
  therefore $t:=\mathrm{dim}\, \pi(Z) \leq d$.
 
 	In order to bound the dimension of $Z$ we are going to calculate a bound for the dimension of the fiber $\pi^{-1}(p)$ for a generic 
  point $p \in \pi(Z)$.  For a sufficiently general smooth point $p \in \pi(Z)^0$ we have that
  \[\pi^{-1}(p) \subset \{p\} \times \{W \in G(d,n) \, | \, W \supset T_p\pi(Z)\} \] 
  If $\pi(Z)$ is of dimension $t$ then by choosing any (linear) direct sum decomposition of $\C^n= E^{n-t} \bigoplus T_p\pi(Z)$
 we get a 1 to 1 correspondence between the set $ \{W \in G(d,n) \, | \, W \supset T_p\pi(Z)\} $ and the set of $d-t$ linear 
 subspaces of $E^{n-t}$, i.e. a Grassmanian $G(d-t,n-t)$ of dimension $(d-t)(n-d)$. 
 Therefore $\mathrm{dim}\, Z \leq t + (d-t)(n-d)$.
  \end{proof}

	In the proof of this result we have seen that the fiber over a non-singular point $p  \in \pi(Z^0)$ is contained in the set of tangent $d-$planes to $\pi(Z)$ at 
	$p$, that is $d-$dimensional linear subspaces $W$ of $\C^n$  such that $W \supset T_p\pi(Z)$. This means, we are looking at a natural generalization of both
	the Nash mo-dification and the conormal space of a germ of singularity $(X,0) \subset (\C^n,0)$ where we consider limiting d-dimensional linear tangent spaces 
	for any $d$ in $\{\mathrm{dim}\,X, \ldots, n-1\}$. Zak considers these spaces in \cite{Zak} in the case of projective varieties and subvarieties of complex tori.

\section{Characterization of $C_d(X)$ inside $\C^n \times G(d,n)$}	

\begin{definition}\label{Conormaldef}
   Let $(X,0) \subset (\C^n,0)$ be a germ of analytic,reduced and irreducible analytic singularity of dimension $k$. For any $d \in \{k,k+1,\ldots, n-1\}$
   define the $d-conormal$ of $X$ by
   \[C_d(X):= \overline{\{ (z,W)  \in X^0 \times G(d,n) \, | \,  T_zX^0 \subset W\}}\]
   where $X^0$ denotes the smooth part of $X$, $G(d,n)$ is the Grassmann variety of $d-$dimensional linear subspaces of $\C^n$ and the bar 
   denotes closure in $X \times G(d,n)$.  
      We will denote by $\nu_d:C_d(X) \to X$ the restriction of the projection to the first coordinate. 
\end{definition}

   Note that for $d=k$ we have that $C_k(X)$ is the Nash modification of $X$ and for $d=n-1$ we recover the usual conormal space of $X$.
   
   \begin{lema} \label{ConormalIntegral}In the setting of definition \ref{Conormaldef} we have that $C_d(X)$  is an analytic space of 
       dimension $k + (d-k)(n-d)$ and $\nu_d: C_d(X) \to X$ is a proper map. Moreover it is an integral subvariety of $\left(\C^n \times G(d,n), \mathcal{H}\right)$.
   \end{lema}
   \begin{proof}
      That $C_d(X)$ is analytic follows from the fact that $X$ is analytic and the incidence condition $T_zX^0 \subset W$ defining the fiber over a 
     smooth point is algebraic. Moreover the map  $\nu_d$ is proper because $G(d,n)$ is compact. Regarding its dimension, it is the same calculation 
      we did in proposition \ref{Cotadimensional}. That is, for any smooth point $z \in X^0$ we have that
       \[\nu_d^{-1}(z) = \{z\} \times \{W \in G(d,n) \, | \, W \supset T_zX^0\} \] 
       and the set in the second factor is a Grassmannian $G(d-k,n-k)$. This implies that for a smooth germ $(\C^k,0) \subset (\C^n,0)$ 
       we have that $C_d(\C^k)$ is isomorphic to $\C^k \times G(d-k,n-k)$ and so if $z$ is a smooth point of $X$  then any point $(z,W) \in 
       \nu_d^{-1}(z)$ is smooth in $C_d(X)$. \\ 
         Finally, recall that by definition, for any point $(z,W) \in \C^n \times G(d,n)$ we have 
          \[ \mathcal{H}(z,W) = W \times T_WG(d,n)\]
          Now, since the map $\nu_d$ is just the restriction of the projection onto the first factor, then the tangent map $D_{(z,W)}\nu_d $
          is also a projection and for any tangent vector $(\vec{u},\vec{v}) \in T_{(z,W)}C_k(X) \subset T_z\C^n \times T_WG(k,n)$ we have that
          \[D_{(z,W)}\nu_k(\vec{u},\vec{v})=\vec{u} \in T_zX \subset W\]
          that is $(\vec{u},\vec{v}) \in \mathcal{H}(z,W)$ and so $C_d(X) $ is an integral subvariety of $\left(\C^n \times G(d,n), \mathcal{H}\right)$.
   \end{proof}
   
   \begin{Theorem} \label{Characterization}Let $Z \subset \C^n \times G(d,n)$ be a reduced, analytic and irreducible subvariety and $X=\pi(Z)$ where $\pi:\C^n \times G(d,n) \to \C^n$ 
     denotes  the projection to $\C^n$.
     If the dimension of $X$ is equal to $t$, then the following statements are equivalent:
     \begin{itemize}
     \item[i)] $Z$ is the d-conormal space of $X\subset \C^n$.
     \item[ii)] $Z$ is an integral subvariety of $\left(\C^n \times G(d,n), \mathcal{H}\right)$ of dimension $t+(d-t)(n-d)$
     \end{itemize}
   \end{Theorem}
   \begin{proof}
       $i) \Rightarrow ii)$ was proved in lemma \ref{ConormalIntegral}. \\ 
        First note that since $X$ is of dimension $t$ and $Z$ is of dimension $t+(d-t)(n-d)$ then the generic
        fiber of $\pi:Z \to X$ is of dimension $(d-t)(n-d)$. Now, let $z$ be a smooth point of $X$, then for any sufficiently general smooth
        point of its fiber $(z,W) \in Z$ we have that 
        \[D_{(z,W)}\pi(T_{(z,W)}Z)=T_zX\]
        Since $Z$ is an integral subvariety we have that $T_{(z,W)}Z\subset W\times T_W G(d,n)$ and so
        $T_zX \subset W$. This implies that the $(d-t)(n-d)$ dimensional fiber $\pi^{-1}(Z)$ is contained in the $(d-t)(n-d)$ dimensional
        variety $\{z\} \times \{W \in G(d,n) \,| \, T_zX \subset W\}$, and so they must be equal.   But this is precisely the definition
        of the $d-$conormal variety $C_d(X)$.
          \end{proof}
          
            Note that when $d=n-1$ then $t+ (d-t)(n-d)=n-1$ and $C_d(X)\subset \C^n \times \check{\P}^{n-1}$ is the usual conormal space of $X$. Moreover, 
         this theorem recovers the characterization of conormal varieties as legendrian subvarieties of $\C^n \times \check{\P}^{n-1}$
         with its canonical contact structure. (See \cite[Section 10.1, pg 91-92]{Pha1})
         
        \begin{corollary} Let $Z$ be an integral subvariety of $\left(\C^n \times G(d,n), \mathcal{H}\right)$ of dimension $d$. Then $Z$ is the Nash modifcation
        of its image in $\C^n$ if and only if for every smooth point $(z,W) \in Z^0$ the tangent space $T_{(z,W)}Z$ is transverse to the subspace $T_WG(d,n)$ 
        of $T_{(z,W)}\left(\C^n \times G(d,n)\right)$.         \end{corollary}
        \begin{proof}
        $ \Rightarrow]$ Note that for any point $(z,W)$ in  $\C^n \times G(d,n)$ the kernel of the differential $D\pi: T_z \C^n \times T_W G(d,n) \to T_z\C^n$ is $T_WG(d,n)$. On the other hand,
        the Nash modification $\nu:\Na X \to X$ is an isomorphism over the smooth part of $X$ so for any smooth point $z_0 \in X^0$ we have that the differential 
           \[D_{(z_0,T_{z_0}X)}\nu: T_{(z_0,T_{z_0}X)} \Na X \to T_{z_0} X\]
           is an isomorphism. Since the map $\nu$ can be realized as the restriction to $\Na X$ of the projection $\pi: \C^n \times G(d,n) \to \C^n$ this 
           implies that $T_{(z_0,T_{z_0}X)} \Na X$ is transverse to $T_WG(d,n)$.\\
        $\Leftarrow]$ We know that the projection $\pi: Z \to X$ is generically a submersion with the kernel of the differential $D_{(z,W)}\pi:T_{(z,W)}Z \to T_zX$ being equal to
        the intersection of $T_{(z,W)}Z$ and $T_WG(d,n)$ , but the transversality condition means that this this intersection is of dimension zero  which implies
        that $T_zX$ and therefore $X$ is of dimension $d$. By theorem \ref{Characterization} this is equivalent to $Z$ being the Nash modification of $X$.
        \end{proof}
        
        \begin{example}
          For a germ of surface $(S,0) \subset (\C^5,0)$  we have the following spaces:
          \begin{align*}
             \mathrm{Nash \, modification}   \;    & \nu: \Na S \to S \,   \mathrm{dimension} \, 2\\
             \mathrm{3-conormal } \;  &  \nu_3: C_3(S) \to  S \,   \mathrm{dimension} \, 4 \\ 
             \mathrm{Conormal}  \;   &  \kappa: C(S) \to S \,  \mathrm{dimension} \, 4
          \end{align*}
             Since  $\Na S \subset S \times G(2,5)$ and $C_3(S) \subset S \times G(3,5)$ it would be interesting to try to use that these two Grassmannians
             are isomorphic to define a morphism $\Na S \to C_3(S)$ making the following diagram commute:
                  \[ \xymatrix{  \Na S \ar[rr] \ar[dr]_\nu & & C_3(S)\ar[dl]^{\nu_3} \\
                        & S & }\]
             This could be a first step to work out a way from the conormal fiber $\kappa^{-1}(0)$ to the Nash fiber $\nu^{-1}(0)$.
        \end{example}

           As a first application of how this d-conormal spaces can be used, we will  characterize Whitney conditions in the Nash modification of $X$ 
           in an analogous way to the characterization in the conormal space $C(X)$ given in \cite[Proposition 1.3.8]{L-T2}.\\
           
                 Consider a germ of analytic, reduced and irreducible singularity $(X,0)\subset (\C^n,0)$ of dimension $d$  such that 
            its singular locus $(Y,0)$ is smooth of dimension $t$. We will fix a coordinate system $(y_1\ldots,y_n,z_{t+1},\ldots,z_n)$ in $\C^n$
            and we can assume that $Y$ is equal to $\C^t \times \{0\}$.\\
            
                 Note that the d-conormal of $C_d(Y) \subset \C^n \times G(d,n)$  of $Y$ in  $\C^n$ is equal to $Y \times \{W \in G(d,n) \, | \, W \supset Y\}$
              and so it is enough to consider the charts $\C^n \times G_d^0(n,W_1)$ of $\C^n \times G(d,n)$ where $W_1$ is a coordinate $n-d$ linear subspace 
              such that $W_1 \cap Y =\{0\}$. \\
              
               	Moreover, after identifying $G_d^0(n,W_1)$ with $\mathrm{Hom}_\C(W_0,W_1)$, we can take 
              $W_0=\C \cdot  \left< e_1,\ldots,e_t,e_{i_{t+1}},\ldots, e_{i_d}\right>$ and in this chart the $W$'s that contain $Y$ correspond to 
              linear morphisms $L:W_0 \to W_1$ such that $Y \subset \mathrm{Ker}(L)$. \\
              
                We will use the fact that in complex analytic geometry Whitney's condition b) is equivalent (\cite[Chap. 5]{Te1}) to condition w) which we now recall.
                The couple $(X^0,Y)$ satisfies condition w) at the origin if there exists an open neighborhood of the origin  $U \subset X$ and a real positive constant $C$
                such that for every $y \in U \cap Y$ and $x \in U \cap X^0$ we have that 
                \[\delta(T_yY,T_xX^0) \leq C d(x,Y)\]
                where $d(x,Y)$ is the euclidean distance in $\C^n$, $\delta$ is defined for linear subspaces $A,B \subset \C^n$  by:
                \[ \delta(A,B):=\sup_{\vec{u} \in B^\perp \setminus \{0\}, \vec{v} \in A \setminus\{0\}} \frac{\left| \left< \vec{u}, \vec{v} \right> \right|}{|| \vec{u}|| \,|| \vec{v}||}\]
                and $\left< \vec{u}, \vec{v} \right>$ denotes the usual hermitian product in $\C^n$.

           \begin{proposition}  Let $\mathcal{I}$ denote the ideal of $\mathcal{O}_{\Na X}$ that defines the intersection
           $C_d(Y) \cap \Na X$ and $J$ the ideal defining $\nu^{-1}(Y)$. 
           \begin{enumerate}
           \item The couple $(X\setminus Y, Y)$ satisfies Whitney's condition $a)$ at the origin if and only if at every point $(0,T)\in \nu^{-1}(0)$
            we have that $\sqrt{\mathcal{I}}=\sqrt{J}$ in $O_{\Na X, (0,T)}$.
           \item The couple $(X\setminus Y, Y)$ satisfies condition $w)$ at the origin if and only if at every point $(0,T)\in \nu^{-1}(0)$ the ideals $\mathcal{I}$ 
                    and  $J$  have the same integral closure in $\mathcal{O}_{\Na X, (0,T)}$.
           \end{enumerate}
            \end{proposition}
           \begin{proof} 
              Note that we always have the inclusion $C_d(Y) \cap \Na X \subset \nu^{-1}(Y)$, or equivalently $ \mathcal{I} \supset J$.\\
              For 1), recall that Whitney's condition a) demands that every limit of tangent spaces $T$ to $X$ at $0$ contains the tangent space to $Y$ at $0$, which
             we can identify with $Y$ since it is linear.  This is exactly what the set-theoretical equality $C_d(Y) \cap \nu^{-1}(0) = \nu^{-1}(0)$ means which is equivalent
             to  $\sqrt{\mathcal{I}}=\sqrt{J}$ in $\mathcal{O}_{\Na X, (0,T)}$ for every  point $(0,T)\in \nu^{-1}(0)$.\\
             
             2) $\Leftarrow ]$
             
                Now suppose that at every point $(0,T)\in \nu^{-1}(0)$ the ideals $\overline{\mathcal{I}}$  and  $\overline{J}$  are equal in $O_{\Na X, (0,T)}$,
             in particular they have the same radical, and so by 1) we have that $Y \subset T$ and  by the discussion prior to the proposition we can see it
             in a chart of $\C^n \times G(d,n)$ of the form $\C^n \times \mathrm{Hom}_\C(W_0,W_1)$, where $W_1$ is an $n-d$ linear coordinate subspace
             transversal to $Y$   and  the $d$  linear subspace   $W_0$ can be taken  of the form  $\C \cdot  \left< e_1,\ldots,e_t,e_{i_{t+1}},\ldots, e_{i_d}\right>$.\\
             
                 In this chart we have a coordinate system \[(y_1,\ldots,y_t,z_{t+1}, \ldots,z_n, a_{ij})  i=1,\ldots, n-d, j=1,\ldots,d\]
                where $J=\left< z_{t+1},\ldots,z_n\right>\mathcal{O}_{\Na X}$ and since $W \in G(d,n) $ contains $Y$ if and only if $Y$ is in the kernel of the corresponding
                linear map $L_W \in \mathrm{Hom}_\C(W_0,W_1)$, that is $L_W(e_i)=\vec{0}$ for $i=1,\ldots,t$ we have that 
               \[\mathcal{I}=\left< z_{t+1},\ldots,z_n, a_{ij} ; i=1,\ldots,n-d; j=1,\ldots t \right>\mathcal{O}_{\Na X}\]
               \[J=\left< z_{t+1},\ldots,z_n\right>\]
               
              The equality of integral closures $\overline{\mathcal{I}} = \overline{J}$ implies that the coordinate functions 
                \[ a_{ij} \in \overline{J}\mathcal{O}_{\Na X, (0,T)}\]
                and by \cite[Thm 2.1]{Lej-Te} this is equivalent to the existence of an open set $V' \subset \Na X$ and a real positive constant $C_{V'}$ such that
                $(0,T) \in V'$ and for every $(p,W) \in V'$ we  have that 
                \[|a_{ij}| \leq C_{V'} \sup \{|z_{t+1}|,\ldots,|z_n|\} \simeq C_{V'} d(p,Y)\]
                Doing this for every point $(0,T) \in \nu^{-1}(0)$ we obtain an open cover of the fiber and since it is compact we can obtain a finite subcover
                \[ \nu^{-1}(0) \subset (V_1,C_1) \cup \cdots \cup (V_r,C_r)\] 
                Note that $U:= \nu(V_1 \cup V_2 \cup \cdots \cup V_r)$ is an open neighborhood of the origin in $X$, and define 
                $C:=\mathrm{max}\{C_1, \ldots, C_r\}$. Now for any smooth point $p \in U \cap X^0$ we have that the point $(p,T_pX^0)$
                 \[|a_{ij}| \leq C_j \sup \{|z_{t+1}|,\ldots,|z_n|\} \leq C \sup \{|z_{t+1}|,\ldots,|z_n|\} \simeq C d(p,Y)\]
                 Now to finish the proof we will show that 
                 \[\ \delta(T_yY,T_pX^0) \leq \left(C t\sqrt{n-d})\right)  d(p,Y)\]
                 Using the local coordinates of the chosen chart it is enough to prove that for any point $(x,W)$ in this chart we have that 
                 \[ \delta(Y,W) \leq t\sqrt{n-d} \sup\left\{|a_{ij}|,\, i=1,\ldots,n, \, j=1,\ldots,t\right\}\]
                 By definition we have 
              \[ \delta(Y,W):=\sup_{\vec{u} \in W^\perp \setminus \{0\}, \vec{v} \in Y \setminus\{0\}} \frac{\left| \left< \vec{u}, \vec{v} \right> \right|}{|| \vec{u}|| \,|| \vec{v}||}\]
                  Now $Y=\C \cdot \left<\hat{e}_1,\ldots,\hat{e}_t\right>$ and $W=\C \cdot \left< (\hat{e}_1,a_{i1}), \ldots, (\hat{e}_d,a_{id})\right>$
                  and using the Hermitian product we  get  the following relations for $\vec{u} \in W^\perp$:
                \begin{align*}
                   0=\left<  (\hat{e}_1,a_{i1}), \vec{u}  \right> &= \overline{u_1} + a_{11}\overline{u_{d+1}}+a_{21}\overline{u_{d+2}}+ \cdots +a_{(n-d)1}\overline{u_{n}}\\
                   0=\left<  (\hat{e}_2,a_{i2}), \vec{u}  \right> &= \overline{u_2} + a_{12}\overline{u_{d+1}}+a_{22}\overline{u_{d+2}}+ \cdots +a_{(n-d)2}\overline{u_{n}}\\
                     \vdots \\
                   0=\left<  (\hat{e}_d,a_{id}), \vec{u}  \right> &= \overline{u_d} + a_{1d}\overline{u_{d+1}}+a_{2d}\overline{u_{d+2}}+ \cdots +a_{(n-d)d}\overline{u_{n}}  
                \end{align*}
                And so we have:
                \begin{align*}
                    \frac{\left| \left< \vec{u}, \vec{v} \right> \right|}{|| \vec{u}|| \,|| \vec{v}||} & = \frac{\left| \left< \vec{u}, \sum_{i=1}^t \lambda_i \hat{e}_i\right> \right|}{
                      || \vec{u}|| \,|| \sum_{i=1}^t \lambda_i \hat{e}_i||}= \frac{\left|   \sum_{i=1}^t \overline{\lambda_i} u_i\right|}{
                      || \vec{u}|| \,|| \sum_{i=1}^t \lambda_i \hat{e}_i||}\\
                       &\leq \frac{   \sum_{i=1}^t \left| \overline{\lambda_i} u_i \right|}{|| \vec{u}|| \,|| \sum_{i=1}^t \lambda_i \hat{e}_i||}
                          \leq  \frac{   \left| \overline{\lambda_1} u_1 \right|}{ || \vec{u}|| \,|| \lambda_1 \hat{e}_1||} + \cdots + 
                                  \frac{   \left| \overline{\lambda_t} u_t \right|}{ || \vec{u}|| \,|| \lambda_t\hat{e}_t||} \\
                        & = \sum_{i=1}^t \frac{u_i}{||\vec{u}||} = \frac{| \sum_{j=1}^{n-d}a_{j1}\overline{u_{d+j}}|}{||\vec{u}||}  + \cdots +   
                            \frac{| \sum_{j=1}^{n-d}a_{jt}\overline{u_{d+j}}|}{||\vec{u}||} \\
                        &\leq || (\underline{0},a_{11},\ldots,a_{(n-d)1}|| + \cdots + ||(\underline{0},a_{1t},\ldots,a_{(n-d)t}|| \\
                        & \leq \sqrt{n-d} \sup\{|a_{11}|,\ldots,|a_{(n-d)1}|\}  + \cdots + \sqrt{n-d} \sup\{|a_{1t}|,\ldots,|a_{(n-d)t}|\} \\
                        & \leq t\sqrt{n-d} \sup\{|a_{ij}|, \, i=1,\ldots,n-d; \, j=1,\ldots,t\}   
                \end{align*}
                
                2) $\Rightarrow ]$ \\
                
                       By hypothesis the couple $(X\setminus Y, Y)$ satisfies condition w) at the origin, and since in complex analytic geometry this condition
                       is equivalent to Whitney conditions, then for every point $(0,T) \in \nu^{-1}(0)$ we have that $Y \subset T$ and so we can restrict ourselves
                       to look at the charts we have been working on. Without loss of generality we will look at the chart $\C^n \times \mathrm{Hom}_\C(W_0,W_1)$
                       with coordinate system
                        \[\left(y_1,\ldots,y_t,z_{t+1},\ldots,z_n,a_{ij}\right); \, i=1\ldots,n-d, \, j=1\ldots,d \]
                       where $W_0=\C \cdot  \left< e_1,\ldots,e_d\right>$ and $W_1=\C \cdot  \left< e_{d+1},\ldots, e_n\right>$.
                       In this coordinate system we have the ideals
                       \begin{align*}
                           J=& \left< z_{t+1},\ldots,z_n\right>\mathcal{O}_{\Na X} \\
                           \mathcal{I}=&\left< z_{t+1},\ldots,z_n, a_{ij} ; i=1,\ldots,n-d; j=1,\ldots t \right>\mathcal{O}_{\Na X}
                       \end{align*}
                      and we want to prove that $\overline{\mathcal{I}}=\overline{J}$ in $\mathcal{O}_{\Na X, (0,T)}$ for every point $(0,T) \in \nu^{-1}(0)$. \\
                      
                        Again by hypothesis we have an open neighborhood of the origin $U \subset X$ and a real positive constant $C$ such that for every smooth point 
                        $p \in U \cap X^0$ 
                        \[ C \sup\left\{ |z_{t+1}|,\ldots, |z_n|\right\} \geq \delta (Y,T_pX^0):=\sup_{\vec{u} \in (T_pX^0)^\perp \setminus \{0\}, \vec{v} \in Y \setminus\{0\}} 
                            \frac{\left| \left< \vec{u}, \vec{v} \right> \right|}{|| \vec{u}|| \,|| \vec{v}||}\]
                         Note that for any $W \in \mathrm{Hom}_\C(W_0,W_1)$ with coordinates $(b_{ij})$ in this chart, using the relations previously obtained, we have 
                         that $\vec{u} \in W^\perp$ if and only if it is of the form: 
                         	  \[  \left( \begin{array}{c}  u_1 \\ u_2 \\ \vdots \\u_d\\ u_{d+1}\\ \vdots\\ \cdot \\ u_n \end{array} \right) =
	                   \lambda_1 \left( \begin{array}{c}  -b_{11} \\ -b_{12} \\ \vdots \\-b_{1d} \\ 1 \\  0 \\ \vdots \\0\end{array} \right) + 
	                    \lambda_2 \left( \begin{array}{c}  -b_{21} \\ -b_{22} \\ \vdots \\-b_{2d} \\ 0 \\  1 \\ \vdots \\0\end{array} \right)+ \cdots +
	                     \lambda_{n-d} \left( \begin{array}{c}  -b_{(n-d)1} \\ -b_{(n-d)2} \\ \vdots \\-b_{(n-d)d} \\ 0 \\  0 \\ \vdots \\1\end{array} \right) \]
	                     with $\lambda_i \in \C$.
                            
                        Fix a point $(0,T_0)$ in the Nash fiber and consider an open neighbourhood $V:=\{(a_{ij}) \in \C^{d(n-d)} \, | \, |a_{ij}|<M \}$
                        where $M$ is a sufficiently big real positive constant. Now for any point $(p,W) \in U \times V$  we have
                        \[ C \sup\left\{ |z_{t+1}|,\ldots, |z_n|\right\} \geq \delta (Y,W):=\sup_{\vec{u} \in W^\perp \setminus \{0\}, \vec{v} \in Y \setminus\{0\}} 
                            \frac{\left| \left< \vec{u}, \vec{v} \right> \right|}{|| \vec{u}|| \,|| \vec{v}||}\]
                        in particular, by setting $\vec{v}= \hat{e}_j$ and $\vec{u}= (-b_{k1},\ldots, -b_{kd},0,\dots,0,1,0,\ldots,0)$  for $j \in \{1,\ldots,t\}$ 
                        and $k \in \{1,\ldots,n-d\}$ we get the inequality           
                        \[ C \sup\left\{ |z_{t+1}|,\ldots, |z_n|\right\} \geq  \frac{\left| \left< \vec{u}, \hat{e}_j \right> \right|}{|| \vec{u}|| \,||\hat{e}_j||}
                        = \frac{|b_{kj}|}{|| \vec{u}||}> \frac{|b_{kj}|}{M'}\]      
                        the last inequality coming from the fact that the $b_{ij}$'s  are bounded since $W$ is in $V$. This implies that for every      
                        $j \in \{1,\ldots,t\}$ and $i \in \{1,\ldots,n-d\}$  we have that $a_{ij} \in \overline{J}$ which finishes the proof.               
         \end{proof}

         As a final comment we would like to point out that the classic construction of the local polar $(P_k(X),0)$ varieties using the Nash modification, or the 
         conormal space (\cite[Chap. 4, Coro 1.3.2 \& Prop 4.1.1]{Te1}) carries over practically word for word to the d-conormal.  \\
         
            Recall that for a germ of reduced and equidimensional complex analytic singularity  $(X,0) \subset (\C^n,0) $ of dimension $d$ and a sufficiently
        general linear space $D$ of dimension $n-d+k-1$ ($k \in \{1,\ldots,d-1\}$)  the polar variety $P_k(X;D) \subset X$ is the closure in $X$ of the critical locus of the 
        linear projection with kernel $D$ 
        \[ \Pi_D: X^0 \to \C^{d-k+1}\]
         It is a reduced analytic variety of dimension $d-k$, with the property that the multiplicity of $(P_k(X;D),0)$ is an analytic invariant of the germ $(X,0)$.\\
         
           Now for any $\ell \in \{d, \ldots, n-1\} $ and $k \in \{1,\ldots,d-1\}$ take the Schubert variety 
               \[c_k(D):=\left\{W \in G(\ell,n) \, | \, \mathrm{dim} \, W \cap D \geq  k+ \ell - d \right\}\]
            and consider the diagram
              \[ \xymatrix{ 
          & C_\ell(X) \subset X \times G(\ell,n)  \ar[dr]_{\gamma} \ar[dl]^{\nu_\ell} &  \\
             X & &G(\ell,n) }\]
             then \begin{enumerate}
     \item $P_k(X;D)=\nu_\ell\left(\gamma^{-1}(c_k(D))\right)$
     \item  The equality \[\mathrm{dim}\,\left( \nu_\ell^{-1}(0) \cap \gamma^{-1}(c_k(D)) \right)= \mathrm{dim}\, \nu_\ell^{-1}(0)-(\ell-d)(n-\ell)-k\] is true if the intersection is not empty.
         \end{enumerate}
         where $(\ell-d)(n-\ell)$ is the dimension of the fiber $\nu_\ell^{-1}(p)$ for any smooth point $p \in X$.

\bibliographystyle{alpha}
\bibliography{bibliothese}

\begin{thebibliography}{LJT08}

\bibitem[LJT08]{Lej-Te}
M.~Lejeune-Jalabert and B.~Teissier.
\newblock {Cl{\^o}ture int{\'e}grale des id{\'e}aux et
  {\'e}quisingularit{\'e}.}
\newblock {\em Annales de la Facult{\'e} des Sciences de Toulouse},
  XVII(4):781--859, 2008.

\bibitem[LT88]{L-T2}
D.T. L{\^e} and B.~Teissier.
\newblock {Limites d'espaces tangents en g{\'e}om{\'e}trie analytique. }.
\newblock {\em Comment. Math. Helv.}, 63(4):540--578, 1988.

\bibitem[Pha79]{Pha1}
Fr{\'e}d{\'e}ric Pham.
\newblock {\em {Singularit{\'e}s des syst{\`e}mes diff{\'e}rentiels de
  Gauss-Manin. Avec des contributions de Lo Kam Chan, Philippe Maisonobe et
  Jean-Etienne Rombaldi.}}
\newblock {Progress in Mathematics. 2. Boston-Basel-Stuttgart: Birkh{\"a}user.
  V, 339 p. SFr. 34.00 }, 1979.

\bibitem[PT08]{Pic}
Paolo Piccione and Daniel~Victor Tausk.
\newblock {\em A student's guide to symplectic spaces, {G}rassmannians and
  {M}aslov index}.
\newblock Instituto de Matem\'atica Pura e Aplicada (IMPA), Rio de Janeiro,
  2008.

\bibitem[Tei82]{Te1}
B.~Teissier.
\newblock {Vari{\'e}t{\'e}s polaires. II: Multiplicit{\'e}s polaires, sections
  planes et conditions de Whitney.}
\newblock {Algebraic geometry, Proc. int. Conf., La R{\'a}bida/Spain 1981,
  Lect. Notes Math. 961, 314-491 (1982).}, 1982.

\bibitem[Zak93]{Zak}
F.L. Zak.
\newblock {\em Tangents and secants of algebraic varieties}, volume 127 of {\em
  Translations of Mathematical Monographs}.
\newblock American Mathematical Society, Providence, RI, 1993.

\end{thebibliography}

\hspace{1in}

{\scriptsize   \noindent Arturo E. Giles Flores\\
                    Universidad Aut\'onoma de Aguascalientes\\
                    Centro de Ciencias B\'asicas\\
                    Departamento de Matem\'aticas y F\'isica\\
                     arturo.giles@cimat.mx }

\end{document}